\documentclass[reqno]{amsart}
\usepackage[margin=1in]{geometry}
\usepackage{xcolor}
\usepackage{enumitem}
\usepackage{amssymb}
\usepackage{clipboard}
\usepackage{graphicx}

\definecolor{bluecite}{HTML}{0875b7}
\usepackage[unicode=true,
bookmarksopen={true},
pdffitwindow=true,
colorlinks=true,
linkcolor=bluecite,
citecolor=bluecite,
urlcolor=bluecite,
hyperfootnotes=false,
pdfstartview={FitH},
pdfpagemode= UseNone]{hyperref}

\newcommand{\defeq}{\stackrel{\textup{def}}{=}}
\newcommand{\diff}{\,\mathrm{d}}

\newtheorem{theorem}{Theorem}[section]
\newtheorem*{theorem*}{Theorem}

\theoremstyle{definition}

\newtheorem{remark}[theorem]{Remark}

\title[Higher-order Sobolev and inequalities]{Higher-order Sobolev and Rellich inequalities\\ via iterated Talenti's principle}
\author{Csaba Farkas}\address{\textsc{Csaba Farkas}: Department of Mathematics and Computer Sciencie, Sapientia Hungarian University of Transylvania, Tg. Mure\c s, Romania}	\email{farkascsaba@uni.sapientia.ro \& farkas.csaba2008@gmail.com}
\author{S\'andor Kaj\'ant\'o}\address{ \textsc{S\'andor Kaj\'ant\'o}: Department of Mathematics and Computer Science, Babe\c s-Bolyai University, Cluj-Napoca, Romania}
\email{sandor.kajanto@ubbcluj.ro}
\date{\today}
\thanks{C. Farkas and S. Kajántó are supported by the János Bolyai Research Scholarship of the Hungarian Academy of Sciences.}


\begin{document}
\begin{abstract}
    In this paper we establish higher-order Sobolev and Rellich-type inequalities on non-compact Riemannian manifolds  supporting an isoperimetric inequality. We highlight two notable settings: manifolds with non-negative Ricci curvature and having Euclidean volume growth (supporting  Brendle's isoperimetric inequality) and manifolds with non-positive sectional curvature (satisfying the Cartan--Hadamard conjecture or supporting Croke's isoperimetric inequality).  Our proofs rely on various symmetrization techniques, the key ingredient is an iterated Talenti's comparison principle. The non-iterated version is analogous to the main result of Chen and Li [J. Geom. Anal., 2023]. 
\end{abstract}
\maketitle

\section{Introduction}\label{sec:intro}
In the last century, \emph{Schwarz symmetrization} became a fundamental tool in proving functional inequalities. The symmetric rearrangement of a set \(\Omega\subset \mathbb{R}^n\), typically denoted by \(\Omega^\star\), is a ball centered at the origin with the property \(\operatorname{vol}(\Omega)=\operatorname{vol}(\Omega^\star)\). In a similar manner, the symmetric rearrangement of a function \(f\colon \Omega\to \mathbb{R}\) is a function \(f^\star\colon \Omega^\star\to [0,\infty)\), whose level sets are the symmetrizations of the level sets of \(|f|\).

When proving functional inequalities, one typically exploits the following two fundamental properties of symmetrizations (see e.g., Lieb and Loss~\cite{lieb2001analysis}): \emph{the Cavalieri principle} (for every \(q>0\), the \(L^q\)-norm is preserved under symmetrization) and the \emph{Pólya--Szegő inequality} (for every \(p>1\), the \(p\)-Dirichlet energy decreases under symmetrization). If in addition, the given functional inequality involves (singular) weights, the toolkit extends by the \emph{Hardy--Littlewood inequality} (the \(L^1\)-norm of the product of two functions increases under symmetrization). For more details we refer to~\S\ref{sec:itp}.

An alternative way to investigate a functional inequality, is to analyze the associated partial differential equation. According to \emph{Talenti's principle}~\cite{talenti1976elliptic}, if \(u\) and \(v\) are  weak solutions of the problems 
\[
    \begin{cases}
        -\Delta u=f,&\text{in }\Omega,\\
        u=0,&\text{on }\partial\Omega,
    \end{cases}
\quad\text{and}\quad
        \begin{cases}
        -\Delta v=f^\star,&\text{in }\Omega^\star,\\
        v=0,&\text{on }\partial\Omega^\star,
    \end{cases}
\]
respectively, then almost everywhere in \(\Omega^\star\) one has \(u^\star\le v\). By an iterative application of this result, in the spirit of Gazzola, Grunau, and Sweers~\cite{gazzola2010optimal}, one can prove higher order functional inequalities, as well.

The common key ingredient in establishing the aforementioned results, is the following sharp \emph{isoperimetric inequality}: If \(\Omega\subset \mathbb{R}^n\) is a \emph{regular domain}, that is a bounded open set with smooth boundary, then 
\[
	    \operatorname{per}(\partial\Omega)\ge n\omega_n^\frac{1}{n}\cdot\operatorname{vol}(\Omega)^\frac{n-1}{n},
\]
where \(\omega_n\) is the volume of the \(n\)-dimensional unit ball, while \(\operatorname{per}(\cdot)\) and \(\operatorname{vol}(\cdot)\) stands for the perimeter and volume, respectively. Moreover,  equality holds if and only if \(\Omega\) is a ball. In other words, the ball has the least perimeter among the sets with  given volume.

The goal of the paper is to show that if an  isoperimetric inequality holds on a Riemannian manifold, then various Euclidean functional inequalities extend naturally to that setting. Throughout the paper, let \((M,g)\) be an \(n\)-dimensional Riemannian manifold, and suppose that there exists a constant \(c_g>0\) such that for every regular domain \(\Omega\subset M\) one has
\begin{equation}\label{eq:intro:isoperimetric}
	 \operatorname{per}(\partial\Omega)\ge c_g\cdot n\omega_n^\frac{1}{n}\cdot\operatorname{vol}(\Omega)^\frac{n-1}{n}.
\end{equation}

Geometric conditions that guarantee the validity of the inequality~\eqref{eq:intro:isoperimetric} with an explicit value of \(c_g\) are the following (we refer to~\S\ref{sec:prelim:isoperimetric} for more details):
\begin{enumerate}
	\item non-negative Ricci curvature and Euclidean volume growth, see Brendle~\cite{brendle2023sobolev};
	\item non-positive sectional curvature, see Croke~\cite{croke1984sharp};
	\item non-parabolicity and bounded elliptic Kato constant, see Impera, Rimoldi, and Veronelli~\cite{impera2025asymptotically}.
\end{enumerate}

In the sequel,  assuming inequality~\eqref{eq:intro:isoperimetric}, we present our main results. Hereafter, for every \(m\ge1\) let
\[
	|D_g^mu|\defeq\begin{cases}
		|\Delta_g^ku|, &\text{if }m=2k,\\
		|\nabla_g(\Delta_g^ku)|, &\text{if }m=2k+1,
	\end{cases}
\]
where \(\Delta^k_g u\) means that the Laplace--Beltrami operator of \((M,g)\) is applied \(k\)-times to \(u\), while \(\nabla_g\) stands for the Riemannian gradient.

Our first result is a higher-order Sobolev inequality, which can be stated as follows.
\begin{theorem}\label{thm:sobolev}
	Let \((M,g)\) be a complete, non-compact \(n\)-dimensional Riemannian manifold supporting~\eqref{eq:intro:isoperimetric}. Let \(m\) be a positive integer and \(p>1\) such that \(n>mp\).  Then for every \(u\in C_0^\infty(M)\) one has 
	\begin{equation}\label{eq:sobolev}
		\int_M|D_g^mu|^p\diff v_g\ge (c_g)^{mp}\cdot S_{m,p}\cdot \left(\int_M |u|^{p^*}\diff v_g\right)^\frac{p}{p^*},
	\end{equation}
	where \(p^*\defeq\frac{np}{n-mp}\), and \(S_{m,p}\) is the Euclidean Sobolev constant of degree \(m\) and order \(p\).
\end{theorem}

Note that if \(p=1\) (and \(m=1\)), then the inequality~\eqref{eq:sobolev} becomes the so-called \(L^1\)-Sobolev inequality, which turns out to be equivalent to the isoperimetric inequality~\eqref{eq:intro:isoperimetric}. According to Talenti~\cite{Talenti1976}, ``if \(p = 1\) the
Sobolev inequality behaves in a slightly diﬀerent manner'', and the determination of the best constant becomes problematic, see e.g.~Cassani, Ruf, and Tarsi~\cite{cassani2010best}. In this paper, we focus on the case when \(p>1\).

The investigation of the best constant in the inequality~\eqref{eq:sobolev} in the Euclidean setting (when \(m=1\)) originates from the pioneering works of Aubin~\cite{Aubin1976} and Talenti~\cite{Talenti1976}. 
For general \(m\) and \(p=2\) it is given by
\[
	S_{m,2}=\left(\pi^mn(n-2m)\prod_{s=1}^{m-1}(n^2-4s^2)\right)^{-1}\left(\frac{\Gamma(n)}{\Gamma(\frac{n}{2})}\right)^\frac{2m}{n}.
\]
see  Cotsiolis and Tavoularis~\cite{cotsiolis2004best}.

The sharp version of the inequality~\eqref{eq:sobolev} on Cartan--Hadamard manifolds (i.e., complete, simply connected Riemannian manifolds having non-positive sectional curvature) is also provided by Aubin, assuming the validity of the so-called Cartan--Hadamard conjecture (that is the sharp isoperimetric inequality on the aforementioned setting); see also Hebey~\cite[Chapter 8]{hebey2000nonlinear}. Concerning spaces with non-positive Ricci curvature,   Ledoux~\cite{Ledoux1999} showed that attaining the Euclidean optimal constant in a Sobolev inequality forces \((M,g)\) to be isometric to the Euclidean space. The sharp version of the inequality~\eqref{eq:sobolev} is established by Brendle~\cite{brendle2023sobolev} using ABP method, and Krist\'aly~\cite{KristalyCalcVar24} using optimal mass transport theory.

Concerning the case when \(m=2\) and \(p=2\),  Barbosa and Kristály~\cite{barbosa2018} proved a rigidity result for the corresponding Sobolev inequality in the same spirit as Ledoux. In Theorem~\ref{thm:sobolev:sharpness} we explicitly determine (under a mild geometric assumption, see Remark~\ref{rem:sharp}) the sharp constant for the aforementioned inequality. The issue of sharpness in the general case (\(m>2\) and \(p>1\)) is challenging even in the Euclidean setting.

Our second result is the following higher-order Rellich inequality:
\begin{theorem}\label{thm:rellich}
	Let \((M,g)\) be a complete, non-compact \(n\)-dimensional Riemannian manifold supporting~\eqref{eq:intro:isoperimetric}. Let \(m\) be a positive integer and \(p>1\) such that \(n>mp\). Fix \(x_0\in M\) and let \(\rho(\cdot)=d_g(x_0,\cdot)\) denote the Riemannian distance from \(x_0\). Then for every \(u\in C_0^\infty(M)\) one has 
	\begin{equation}\label{eq:rellich}
		\int_M|D_g^mu|^p\diff v_g\ge (c_g)^{mp}\cdot R_{m,p}\cdot \int_M \frac{|u|^{p}}{\rho^{mp}}\diff v_g,
	\end{equation}
	where \(R_{m,p}\) is the Euclidean Rellich constant of degree \(m\) and order \(p\).
\end{theorem}
The investigation of the best constant in the inequality~\eqref{eq:rellich} in the Euclidean setting is due to Mitidieri~\cite{mitidieri2000simple}, the exact value is given by
\[
	R_{m,p}=\begin{cases}
		\displaystyle\prod_{s=1}^k\left(\frac{n}{p}-2s\right)^p\left(\frac{n(p-1)}{p}+2s-2\right)^p,&\text{if } m=2k,\\
		\displaystyle\left(\frac{n-p}{p}\right)^p\prod_{s=1}^k\left(\frac{n}{p}-2s-1\right)^p\left(\frac{n(p-1)}{p}+2s-1\right)^p,&\text{if } m=2k+1.
	\end{cases}
\]

Since Hardy--Rellich inequalities typically does not have extremizers, they admit various improvements. For a comprehensive discussion on them we refer to the monograph by Ghoussoub and Moradifam~\cite{ghoussoub2013functional}, we also refer to Kajántó, Krist{\'a}ly, Peter, and Zhao~\cite{riccatipair2023} for an alternative approach to Hardy-type inequalities. 
The sharpness of inequality~\eqref{eq:rellich} can be established also on Cartan--Hadamard manifolds, see Farkas, Kajántó, and Krist{\'a}ly~\cite{farkas2023sharp}. On spaces with non-positive Ricci curvature the issue of sharpness is still open. 

Our third result is an extension of the previous theorem, which allows for discussing improved Rellich inequalities. 

\begin{theorem}\label{thm:general}
Let \((M,g)\) be a complete, non-compact \(n\)-dimensional Riemannian manifold supporting~\eqref{eq:intro:isoperimetric}. Let \(m\) be a positive integer and \(p>1\) such that \(n>mp\). 
 Fix \(\Omega\subset M\), \(x_0\in \Omega\) and denote by \(\rho(\cdot)=d_g(x_0,\cdot)\)  the Riemannian distance from \(x_0\). Suppose that \(h\colon[0,\infty)\to[0,\infty)\) is a non-increasing function. Then for every \(u\in C_0^\infty(\Omega)\) the following inequality holds:
\[
	\int_{\Omega}|D_{g}^mu|^p\diff v_{g}\ge (c_g)^{mp}\cdot  C_{m,p}\cdot \int_{\Omega} |u|^p h(\rho)\diff v_{g},
\]
provided that 
\[
	\int_{\Omega^\star}|D_{g_0}^mu^\star|^p\diff v_{g_0}\ge C_{m,p}\cdot \int_{\Omega^\star} |u^\star|^p h(|x|)\diff v_{g_0}.
\]
\end{theorem}

The main technical ingredient underlying the above results is an iterated Talenti's principle,  established on Riemannian manifolds verifying the isoperimetric inequality~\eqref{eq:intro:isoperimetric}, see Theorem~\ref{thm:itp}. The proof of the non-iterated version (see Theorem~\ref{thm:chenli}) is analogous to the main result of Chen and Li~\cite{chen2023talenti}, which provides a Talenti's principle on Riemannian manifolds having \(\mathbf{Ric}_g\ge 0\) and Euclidean volume growth.

The paper is structured as follows. In~\S\ref{sec:itp} we present an iterated Talenti's principle. First, in~\S\ref{sec:prelim:isoperimetric} we recall various isoperimetric inequalities. Next, in~\S\ref{sec:prelim:Schwarz} we display the definition and relevant properties of Schwarz symmetrization in the Riemannian setting. Finally, in~\S\ref{sec:prelim:talenti} we formulate the principle together with its iterated version (see Theorem~\ref{thm:itp}). 
In~\S\ref{sec:sobolev} we prove Theorem~\ref{thm:sobolev}, that is we establish higher-order Sobolev inequalities on Riemannian manifolds supporting the inequality~\eqref{eq:intro:isoperimetric}. Moreover, in Theorem~\ref{thm:sobolev:sharpness} in the  particular case of non-positive Ricci curvature, Euclidean volume growth, and \(m=p=2\), we present a sharpness result. In~\S\ref{sec:rellich} we prove Theorem~\ref{thm:rellich}, that is we establish higher-order Rellich inequalities on Riemannian manifolds supporting the inequality~\eqref{eq:intro:isoperimetric}. Here we also prove Theorem~\ref{thm:general}, which allows for establishing improved higher-order Rellich-type inequalities. We conclude the paper with various applications of Theorem~\ref{thm:general}.

\section{An Iterated Talenti's principle}\label{sec:itp} 
In this section we present an iterated version of Talenti's principle on Riemannian manifolds supporting the isoperimetric inequality~\eqref{eq:intro:isoperimetric}. First,  in~\S\ref{sec:prelim:isoperimetric} we recall some well-known Riemannian isoperimetric inequalities. Next, in~\S\ref{sec:prelim:Schwarz} we present various concepts concerning the Schwarz symmetrization. Finally, in~\S\ref{sec:prelim:talenti} we present an iterated Riemannian version of Talenti's principle. 

\subsection{Isoperimetric inequalities}\label{sec:prelim:isoperimetric} Throughout the paper we consider Riemannian manifolds \((M,g)\) that supports the isoperimetric inequality~\eqref{eq:intro:isoperimetric}. To motivate this assumption, in what follows, we recall some well-known isoperimetric inequalities (together with their relevant properties) in the Riemannian setting.
 
\medskip{\bf Case 1: \(\mathbf{Ric}_g\ge0\)}. The \emph{asymptotic volume ratio} of \((M,g)\) is given by 
\[
    \mathsf{AVR}_g \defeq \lim_{r\to \infty}\frac{\operatorname{vol}_g(B_g(x,r))}{\omega_nr^n},
\]
where \(\operatorname{vol}_g\) stands for the Riemannian volume induced by \(g\); \(B_g(x,r)\) denotes a geodesic ball in \(M\) centered at \(x\in M\) having radius \(r\); while \(\omega_n\) denotes the volume of the \(n\)-dimensional Euclidean unit ball.
On the one hand, we note that the above limit is independent of the choice of \(x\in M\). On the other hand, according to the Bishop--Gromov comparison theorem (see e.g.,~Gallot, Hulin, and Lafontaine~\cite[Theorem 3.101]{gallot2004riemmanian}), one has 
\(\mathsf{AVR}_g\in [0,1]\). The typical assumption in this setting is that the manifold \((M,g)\) has \emph{Euclidean volume growth}, that is
\[
\mathsf{AVR}_g>0.	
\]
Under this assumption the following sharp \emph{isoperimetric inequality} holds (see Brendle~\cite[Corollary 1.3]{brendle2023sobolev}): For every regular domain \(\Omega\subset M\) one has 
\begin{equation}\label{eq:prelim:isoperimetric}
    \operatorname{per}_g(\partial\Omega)\ge \mathsf{AVR}_g^\frac{1}{n}\cdot n\omega_n^\frac{1}{n}\cdot\operatorname{vol}_g(\Omega)^\frac{n-1}{n},
\end{equation}
where  \(\operatorname{per}_g(\partial\Omega)\) and \(\operatorname{vol}_g(\Omega)\)  denote the perimeter and  volume of \(\Omega\), respectively. Moreover, equality occurs if and only if \((M,g)\) is isometric to the Euclidean space and \(\Omega\) is isometric to an  Euclidean ball. For an alternative proof (using optimal mass transport theory) of the above result see Balogh and Kristály~\cite{balogh2023sharp}.

Note that in case of the standard Euclidean space,  denoted in the sequel by \((\mathbb{R}^n,g_0)\), one has \(\mathsf{AVR}_{g_0}=1\), and thus, inequality~\eqref{eq:prelim:isoperimetric} reduces to the Euclidean isoperimetric inequality. We also notice that inequality~\eqref{eq:prelim:isoperimetric} is a particular case of inequality~\eqref{eq:intro:isoperimetric} for \(c_g=\mathsf{AVR}^{1/n}\).\medskip

{\bf Case 2: \(\mathbf{K}_g\le0\)}. This case is strongly related to the so called Cartan--Hadamard conjecture, which can be stated as follows. If \((M,g)\) is a Cartan--Hadamard manifold, that is a complete, simply connected Riemannian manifold with non-positive sectional curvature, then it supports the Euclidean isoperimetric inequality (i.e., inequality~\eqref{eq:intro:isoperimetric} with constant \(c_g=1\)). The conjecture is proved for \(n\in\{2,3,4\}\), see Bol~\cite{bol1941isoperimetrische}, Kleiner~\cite{kleiner1992isoperimetric}, and Croke~\cite{croke1984sharp}, for higher dimensions it is still open. Fortunately, according to Croke  the isoperimetric inequality~\eqref{eq:intro:isoperimetric} holds for \(n\ge 3\) with the constant
\begin{equation}\label{eq:croke}
	c_g=n^{-\frac{1}{n}}\omega_n^{1-\frac{2}{n}}\left((n-1)\omega_{n-1}\int_0^\frac{\pi}{2}\cos^\frac{n}{n-2}(t)\sin^{n-2}(t)\diff t\right)^{\frac{2}{n}-1}.
	\end{equation} 
Note that \(c_g\le 1\), for every \(n\ge 3\), while \(c_g=1\) if and only if \(n=4\).\medskip

{\bf Case 3: Bounded elliptic Kato constant.} As in Impera, Rimoldi, and Veronelli~\cite[Theorem 1.1]{impera2025asymptotically}, let \((M,g)\) be a complete non-parabolic Riemannian manifold
with \[
	\limsup_{r\to \infty}\frac{\operatorname{vol}_g(B_g(y,r))}{r^n}\ge\beta>0.
\]
Suppose that the \emph{elliptic Kato constant} satisfies \(k_\infty<\frac{1}{n-2}\) and \(\Omega\subset M\) is a compact domain with smooth boundary. Then isoperimetric inequality~\eqref{eq:intro:isoperimetric} holds with 
\[
	c_g=(1-(n-2)k_\infty)^\frac{4(n-1)}{n(n-2)}\beta^\frac{1}{n}\omega_n^{-\frac{1}{n}}.
\]

\begin{remark}
	We note that there exist various results ensuring the validity of the isoperimetric inequality~\eqref{eq:intro:isoperimetric}, without specifying the explicit value of the constant \(c_g\). For example, according to Carron~\cite{carron1994inegalites}, Coulhon and Ledoux~\cite{coulhon1994isoperimetrie}, and Varopoulos~\cite{varopoulos1985hardy}, if \((M,g)\) is a complete Riemannian manifold with infinite volume, \(n\ge 3\), \(\mathbf{Ric}_g\ge 0\), then inequality~\eqref{eq:intro:isoperimetric} holds for some \(c_g>0\) if and only if there exists \(K>0\) such that for every \(x\in M\) and \(t>0\) one has 
	\[
		\operatorname{vol}_g(\{y\in M:G_x(y)>t\})\le K\cdot t^{-n/(n-2)},
	\]
	where \(G_x\) is the positive minimal Green function of pole \(x\). 
\end{remark}

\subsection{Schwarz symmetrization}\label{sec:prelim:Schwarz} It turns out that the isoperimetric inequality~\eqref{eq:intro:isoperimetric} enables us to compare various quantities on a general Riemannian manifold to the corresponding quantities on the Euclidean space. This is precisely done using the \emph{symmetrization procedure} presented below.

If \(\Omega\subset M\) is a regular domain, its \emph{symmetrization}, denoted by \(\Omega^\star\), is an Euclidean ball, centered at the origin satisfying 
\begin{equation}\label{eq:prelim:sym:domain}
  \operatorname{vol}_g(\Omega)=\operatorname{vol}_{g_0}(\Omega^\star).  
\end{equation}

Using the Euclidean isoperimetric inequality (with \(c_{g_0}=1\)), inequality~\eqref{eq:intro:isoperimetric} can be rewritten as 
\begin{equation}\label{eq:prelim:iso:rev}
    \operatorname{per}_g(\partial\Omega)\ge c_g\cdot\operatorname{per}_{g_0}(\partial\Omega^\star).
\end{equation}

The \emph{Schwarz symmetrization} of a function \(u\colon\Omega\to\mathbb{R}\), is a function \(u^\star\colon\Omega^\star\to[0,\infty)\), which satisfies 
\begin{equation}\label{eq:prelim:fsym}
     \operatorname{vol}_{g}(\{x\in M:|u(x)|>t\})=\operatorname{vol}_{g_0}(\{x\in \mathbb{R}^n:u^\star(x)>t\}),\quad \forall t>0.
\end{equation}

In the sequel, we present three fundamental properties of symmetrized functions, which are exploited in the rest of the paper. First, the Cavalieri principle (see e.g.~Lieb and Loss~\cite[Theorem~3.3/(iv)]{lieb2001analysis}) states that if \(u\in L^q(\Omega)\) for some \(q\in(0,\infty)\), then 
\begin{equation}\label{eq:prelim:Cavalieri}
    \|u\|_{L^q(\Omega)}=\|u^\star\|_{L^q(\Omega^\star)}.
\end{equation}

Next, according to the Pólya--Szegő inequality (see e.g.~Balogh and Kristály~\cite[Proposition 3.1]{balogh2023sharp}), if  \(|\nabla_g u|\in L^p(\Omega)\) for some \(p>1\) and some smooth function \(u\), then one has
\begin{equation}\label{eq:prelim:PolyaSzego}
	  \|\nabla_g u\|_{L^p(\Omega)}\ge c_g\cdot \|\nabla_{g_0} u^\star\|_{L^p(\Omega^\star)},
\end{equation}
where \(\nabla_g\) and \(\nabla_{g_0}\) denote the Riemannian and Euclidean gradient, respectively. 

Finally, due to the Hardy--Littlewood inequality (see e.g.~Lieb and Loss~\cite[Theorem 3.4]{lieb2001analysis}), if \(x_0\in\Omega\) and \(\rho(x)= d_g(x_0,x)\), \(u\in L^p(\Omega)\) for some \(p>1\), and \(h\colon[0,\infty)\to[0,\infty)\) is a non-increasing function, then 
\begin{equation}\label{eq:prelim:HardyLittlewood}
	\int_\Omega u(x)^ph(\rho(x))\diff v_g\le \int_{\Omega^\star} u^\star(x)^ph(|x|)\diff v_{g_0}.
\end{equation}

\begin{remark}\label{rem:alternativesym}

We note that some authors condiser an alternative symmetrization: \(\Omega^\sharp\) is an Euclidean ball satisfying \(\operatorname{vol}_g(\Omega)=(c_g)^n\cdot \operatorname{vol}_{g_0}(\Omega^\sharp)\) instead of~\eqref{eq:prelim:sym:domain}, and inequality~\eqref{eq:prelim:iso:rev} becomes \(\operatorname{per}_g(\partial\Omega)\ge (c_g)^n\cdot \operatorname{per}_{g_0}(\partial\Omega^\sharp)\).
It turns out that when proving the alternative versions of the fundamental properties~\eqref{eq:prelim:Cavalieri},~\eqref{eq:prelim:PolyaSzego}~\&~\eqref{eq:prelim:HardyLittlewood}, one needs to compare sets with the same volume, and then to perform changes of variables of the form \(y=c_g\cdot x\). Using the \(\Omega^\star\) symmetrization, the proofs can be done is the same spirit as their Euclidean counterparts, and no change of variables is required. 
\end{remark}

\subsection{An iterated Talenti's principle}\label{sec:prelim:talenti} The following result is a Talenti's principle on Riemannian manifolds supporting the isoperimetric inequality~\eqref{eq:intro:isoperimetric}. The statement is formulated using the notations of the previous section. The proof is the same as the one of Chen and Li~\cite[Theorem 1.1]{chen2023talenti}, which exploits the isoperimetric inequality~\eqref{eq:prelim:isoperimetric}, and is therefore omitted.

\begin{theorem}\label{thm:chenli}
Let \((M,g)\) is complete, non-compact Riemannian manifold  supporting inequality~\eqref{eq:intro:isoperimetric}. Assume that \(\Omega\subset M\) is a regular domain and let \(\Omega^\star\) denote its symmetrization. Suppose that \(u\) and \(v\) are weak solutions of the problems
\begin{equation}\label{eq:prelim:chenli:problems}
	\begin{cases}
        -\Delta_gu=f,&\text{in }\Omega,\\
        u=0,&\text{on }\partial\Omega,
    \end{cases}\quad\text{and}\quad\begin{cases}
        -\Delta_{g_0} v=f^\star,&\text{in }\Omega^\star,\\
        v=0,&\text{on }\partial\Omega^\star,
    \end{cases}
\end{equation}
respectively. Then for almost every \(x\in\Omega^\star\) one has 
\begin{equation}\label{eq:prelim:chenli:ineq}
	  (c_g)^2\cdot u^\star(x)\le v(x).
\end{equation}
\end{theorem}

Next, in the spirit of Gazzola, Grunau, and Sweers~\cite{gazzola2010optimal}, we establish an iterated version of Theorem~\ref{thm:chenli}, which enables us to prove higher-order Sobolev and Rellich inequalities. Our result reads as follows.  
\begin{theorem}\label{thm:itp}
	Let \((M,g)\) be a complete, non-compact Riemannian manifold  supporting inequality~\eqref{eq:intro:isoperimetric}.  Assume that \(\Omega\subset M\) is a regular domain and let \(\Omega^\star\) denote its symmetrization. Let \(k\ge 1\) be an integer, and suppose that \(u\) and \(v\) are weak solutions of the problems
	\begin{equation}\label{eq:itp:pde}
	\begin{cases}
		(-\Delta_g)^ku=f,&\text{in }\Omega,\\
        u=\Delta_g u = \dots =\Delta_g^{k-1} u=0,&\text{on }\partial\Omega,
	\end{cases}
\end{equation}
and
\begin{equation}\label{eq:itp:pde:symm}
	\begin{cases}
		(-\Delta_{g_0})^kv=f^\star,&\text{in }\Omega^\star,\\
         v=\Delta_{g_0} v = \dots =\Delta_{g_0}^{k-1} v=0,&\text{on }\partial\Omega^\star.
	\end{cases}	
\end{equation}
Then for almost every \(x\in\Omega^\star\) one has 
\begin{equation}\label{eq:itp:ineq}
	 (c_g)^{2k}\cdot u^\star(x)\le v(x).
\end{equation}
\end{theorem}
\begin{proof}
	We proceed by finite induction. When \(k=1\), Theorem~\ref{thm:itp} follows from Theorem~\ref{thm:chenli}. In the sequel, we prove Theorem~\ref{thm:itp} for an arbitrary \(k\ge 2\), assuming that it is true for \(k-1\). Let
	\[\widetilde u\defeq -\Delta_g u\quad\text{and}\quad\widetilde v\defeq-\Delta_{g_0}v.\]
	
	On the one hand, according to problems~\eqref{eq:itp:pde}~\&~\eqref{eq:itp:pde:symm}, one has 
	\[
		\begin{cases}
		(-\Delta_g)^{k-1}\widetilde u=f,&\text{in }\Omega,\\
        \widetilde u=\Delta_g \widetilde u = \dots =\Delta_g^{k-2} \widetilde u=0,&\text{on }\partial\Omega,
	\end{cases}\quad\text{and}\quad 	
	\begin{cases}
		(-\Delta_{g_0})^{k-1}\widetilde v=f^\star,&\text{in }\Omega^\star,\\
         \widetilde  v=\Delta_{g_0} \widetilde v = \dots =\Delta_{g_0}^{k-2} \widetilde v=0,&\text{on }\partial\Omega^\star.
	\end{cases}
	\]
	Hence, our assumption (Theorem~\ref{thm:itp} is true for \( k-1\)) implies that for a.e.~\(x\in \Omega\) one has \[(c_g)^{2(k-1)}\cdot\widetilde{u}^\star(x)\le\widetilde{v}(x).\]

	On the other hand, we also have \({u}=0\) and \({v}=0\) on \(\partial\Omega\) and \(\partial \Omega^\star\), respectively, thus we obtain
	\[
		\begin{cases}
			-\Delta_gu=\widetilde{u}&\text{in }\Omega,\\
        u=0,&\text{on }\partial\Omega,
		\end{cases}\quad\text{and}\quad
		\begin{cases}
			-\Delta_{g_0}v=\widetilde{v}&\text{in }\Omega^\star,\\
			 v=0,&\text{on }\partial\Omega^\star.
		\end{cases}
	\]

	Finally, let \(z\) be a weak solution of   
	\[
		\begin{cases}
			-\Delta_{g_0}z=\widetilde{u}^\star,&\text{in }\Omega^\star,\\
        z=0,&\text{on }\partial\Omega^\star.
		\end{cases}
	\]
	By Theorem~\ref{thm:chenli}, we obtain that for a.e.\ \(x\in\Omega\) one has \[(c_g)^2\cdot u^\star(x)\le z(x),\] and the following inequality holds in the weak sense: \[-\Delta_{g_0} \left((c_g)^{2(k-1)}\cdot z\right)=(c_g)^{2(k-1)}\cdot\widetilde{u}^\star\le \widetilde v=-\Delta_{g_0} v.\] 
	Thus, the Euclidean maximum principle implies the inequality~\eqref{eq:itp:ineq}.

\end{proof}

  \begin{remark}
	We note that Fontana, Morpurgo, and Qin~\cite{fontana2024sharp} used a version of the above iterated Talenti's principle for Green functions (and Navier boundary conditions) on Riemannian manifolds having Euclidean volume growth.
 	We also  note that the non-iterated version of Talenti's principle is available on Finsler manifolds with non-negative weighted Ricci curvature (see Mester~\cite{mestertalenti}) and on the so-called \(\mathsf{RCD}(0,N)\) spaces (see Wu~\cite{wu2025almost}).  For simplicity, our results are presented on Riemannian manifolds, but could be extended to the aforementioned settings. 
 \end{remark}

\section{Sobolev inequalities}\label{sec:sobolev}
The goal of this section is twofold. First, we prove Theorem~\ref{thm:sobolev}, that is, we establish higher-order Sobolev inequalities on Riemannian manifolds supporting inequality~\eqref{eq:intro:isoperimetric}. Next, we present a sharpness result concerning a particular Sobolev inequality, see Theorem~\ref{thm:sobolev:sharpness}.
\begin{proof}[Proof of Theorem~\ref{thm:sobolev}] Since \(u\in C_0^\infty(M)\), there exists a regular domain \(\Omega \in M\), such that \(u\) (and its derivatives) vanish outside of \(\Omega\). Thus, it is enough to prove the Sobolev inequality~\eqref{eq:sobolev} on an \emph{arbitrary} regular  domain \(\Omega\subset M\). We distinguish two cases.

	\textbf{Case 1:} \(m=2k\). Let \(f\colon \Omega\to \mathbb{R}\) defined by 
	%
	\begin{equation}\label{eq:sobolev:fdef}
	f=(-\Delta_g^k)u.
	\end{equation}
	If follows that, \(u\) is a solution of problem~\eqref{eq:itp:pde}, and let \(v\) be the solution of the symmetrized problem~\eqref{eq:itp:pde:symm}. By the iterated Talenti's principle (see Theorem~\ref{thm:itp}) one has for almost every \(x\in\Omega\) that
	\begin{equation}\label{eq:sobolev:itp:ineq}
		(c_g)^{2k}\cdot u^\star(x)\le v(x), 
	\end{equation}

	First, relations~\eqref{eq:itp:pde}~\&~\eqref{eq:itp:pde:symm} and the Cavalieri principle (see relation~\eqref{eq:prelim:Cavalieri}) implies
	\begin{align*}
		\int_\Omega|D_g^mu|^p\diff v_g=\int_\Omega|\Delta_g^ku|^p\diff v_g=\int_\Omega |f|^p\diff v_g=\int_{\Omega^\star} (f^\star)^p\diff v_{g_0}=\int_{\Omega^\star}|\Delta_{g_0}^kv|^p\diff v_{g_0}.
	\end{align*}

	Next, the Euclidean Sobolev inequality and relation~\eqref{eq:sobolev:itp:ineq} yields
	\[
		\int_{\Omega^\star}|\Delta_{g_0}^kv|^p\diff v_{g_0}\ge 
		S_{m,p}\cdot\left(\int_{\Omega^\star}|v|^{p^*}\diff v_{g_0}\right)^\frac{p}{p^*}\ge 
		(c_g)^{2kp}\cdot S_{m,p}\cdot\left(\int_{\Omega^\star}(u^\star)^{p^*}\diff v_{g_0}\right)^\frac{p}{p^*}.
	\]

	Finally, using the relation \(2k=m\) and the Cavalieri principle leads us to 
	\[
		(c_g)^{2kp}\cdot S_{m,p}\cdot\left(\int_{\Omega^\star}(u^\star)^{p^*}\diff v_{g_0}\right)^\frac{p}{p^*}=(c_g)^{mp}\cdot S_{m,p}\cdot\left(\int_{\Omega}|u|^{p^*}\diff v_{g}\right)^\frac{p}{p^*},
	\]
	which is precisely inequality~\eqref{eq:sobolev}.
	
	\textbf{Case 2:} \(m=2k+1\). Define \(f\) as in relation~\eqref{eq:sobolev:fdef}, then \(u\) is a solution of the problem~\eqref{eq:itp:pde} and let \(v\) be a solution of the symmetrized problem~\eqref{eq:itp:pde:symm}. Note that, here we use the same problem as in the previous case (characterized by the parameter \(k\)).

	On the one hand, using relations~\eqref{eq:itp:pde}~\&~\eqref{eq:itp:pde:symm} and the Pólya--Szegő inequality (see relation~\eqref{eq:prelim:PolyaSzego}) we obtain
	\begin{align*}
		\int_\Omega|D_g^mu|^p\diff v_g&=\int_\Omega|\nabla_g(\Delta_g^ku)|^p\diff v_g=\int_\Omega|\nabla_gf|^p\diff v_g\\&\ge (c_g)^{p}\cdot \int_{\Omega^\star} |\nabla_{g_0}f^\star|^p\diff v_{g_0}= (c_g)^{p}\cdot \int_{\Omega^\star}|\nabla_{g_0}(\Delta_{g_0}^kv)|^p\diff v_{g_0}.
	\end{align*}

	On the other hand, the Euclidean Sobolev inequality, relation~\eqref{eq:sobolev:itp:ineq}, and the Cavalieri principle yields
	\begin{align*}
		(c_g)^{p}\cdot \int_{\Omega^\star}|\nabla_{g_0}(\Delta_{g_0}^kv)|^p\diff v_{g_0}&
		\ge(c_g)^{p}\cdot S_{m,p}\cdot \left(\int_{\Omega^\star}|v|^{p^*}\diff v_{g_0}\right)^\frac{p}{p^*}\\
		&\ge (c_g)^{(2k+1)p}\cdot S_{m,p}\cdot\left(\int_{\Omega^\star}(u^\star)^{p^*}\diff v_{g_0}\right)^\frac{p}{p^*}\\
		&=(c_g)^{mp}\cdot S_{m,p}\cdot\left(\int_{\Omega}|u|^{p^*}\diff v_{g}\right)^\frac{p}{p^*},
	\end{align*}
	which ends the proof.
\end{proof}

We conclude this section with a sharpness result concerning the Sobolev inequality~\eqref{eq:sobolev}  on manifolds with 
\(\mathbf{Ric}_g\ge0\) and Euclidean volume growth, in the particular case when \(m = 2\) and \(p = 2\). This finding extends a result of  Barbosa and Krist\'aly \cite{barbosa2018}: If a second-order Sobolev inequality with the optimal Euclidean constant holds on a Riemannian manifold 
\((M,g)\) with non-negative Ricci curvature, then  \((M,g)\) is isometric to the Euclidean space \((\mathbb{R}^n, g_0)\).

\begin{theorem}\label{thm:sobolev:sharpness}
	Let \((M,g)\) be a complete, non-compact \(n\)-dimensional Riemannian manifold, with \(n> 4\), 
	\(\mathbf{Ric}_g\ge0\), and Euclidean volume growth. Let \(\rho(\cdot)=d_g(x_0,\cdot)\) for some \(x_0\in M\) and suppose that \(\rho\Delta_g\rho\ge n-5\). Then the inequality
	\begin{equation}\label{eq:sobolev:sharp}
		\int_M|\Delta_gu|^2\diff v_g\ge \mathsf{AVR}_g^{\frac{4}{n}}\cdot S_{2,2}\cdot \left(\int_{M} |u|^{\frac{2n}{n-4}}\diff v_g\right)^{\frac{n-4}{n}},\quad\forall u\in C_0^\infty(M),
	\end{equation}
	is \emph{sharp} in the sense that the constant \(\mathsf{AVR}_g^{\frac{4}{n}}\cdot S_{2,2}\) cannot be improved.
\end{theorem}

\begin{proof} The first part of the proof is given in Barbosa and Kristály~\cite[Theorem 1.1/(ii)]{barbosa2018}, we reproduce it for the reader's convenience, the second part can be done in the spirit of Kristály~\cite[Proposition 2.3 \& Theorem 1.1]{KristalyCalcVar24}.

	Suppose that there exists a constant \(C>\mathsf{AVR}_g^{\frac{4}{n}}\cdot S_{2,2}\) such that for every \(u\in C_0^\infty(M)\) one has
	\begin{equation}\label{eq:sobolev:C}
		\int_{M} (\Delta_g u)^2 \diff v_g\ge C\cdot\left(\int_{M} |u|^{\frac{2n}{n-4}}\diff v_g\right)^{\frac{n-4}{n}}.	
	\end{equation}
	To derive a contradiction, we use the \emph{Talentian bubbles} defined by
	\[
		u_\lambda\defeq t_\lambda(\rho),\quad \text{where}\quad t_\lambda(r)\defeq(\lambda+r^2)^{\frac{4-n}{2}}, \quad \lambda,r>0,
	\]
	as test functions in inequality~\eqref{eq:sobolev:C}; this is allowed, since they can be be approximated by \(C_0^\infty(M)\) functions.

	On the one hand, using the chain rule and the eikonal equation (\(|\nabla_g\rho|=1\) a.e.~in \(M\)), for almost every \(x\in M\) we obtain 
	\begin{equation}\label{eq:sobolev:laplace}
		(\Delta_gu_{\lambda})^2=(n-4)^2(\lambda+\rho^2)^{-n}\left[\lambda+(3-n)\rho^2+(\lambda+\rho^2)\rho\Delta_g\rho\right]^2.
	\end{equation}
	
	On the other hand, by assumption and the Laplacian comparison theorem applied for the case \(\mathbf{Ric}_g\ge0\) (see e.g.~Gallot, Hulin, and Lafontaine~\cite[Theorem 3.101]{gallot2004riemmanian}) we also have 
	\[
		n-1\ge \rho\Delta_g\rho\ge n-5,
	\]
	which yields
	\begin{equation}\label{eq:sovolev:sharp:pointwise}
		\left[\lambda+(3-n)\rho^2+(\lambda+\rho^2)\rho\Delta_g\rho\right]^2\leq (n\lambda+2\rho^2)^2. 
	\end{equation}

	According to the relations~\eqref{eq:sobolev:C},~\eqref{eq:sobolev:laplace},~and~\eqref{eq:sovolev:sharp:pointwise},  for every \(\lambda>0\) we obtain
	\begin{align*}
		\nonumber(n-4)^2\int_M(\lambda+\rho^2)^{-n}(n\lambda+2\rho^2)^2\diff v_g
		\ge
		C\cdot\left(\int_{M} (\lambda+\rho^2)^{-n}\diff v_g\right)^{\frac{n-4}{n}}.
	\end{align*}
	Multiplying by \(\lambda^\frac{n-4}{2}\) and using the notation 
	\[
		\mathcal{H}(\lambda,s)\defeq\int_{M} (\lambda+d_g^2)^{-s}\diff v_g,\quad \lambda,s>0,
	\]
	leads us to 
	\begin{equation}\label{eq:sobolev:after}
		(n-4)^2\left(4\lambda^{\frac{n}{2}-2}\mathcal{H}(\lambda,n-2)+4(n-2)\lambda^{\frac{n}{2}-1}\mathcal{H}(\lambda,n-1)+(n-2)^2\lambda^\frac{n}{2}\mathcal{H}(\lambda,n)\right)\ge C\cdot \left(\lambda^\frac{n}{2}\mathcal{H}(\lambda,n)\right)^\frac{n-4}{n}.
	\end{equation}
	
	As a final step we take the limit as \(\lambda \to \infty\) in inequality~\eqref{eq:sobolev:after}. To do this, we use the following asymptotic property of \(\mathcal{H}\) for {\(s>n/2\)}:
	\[
		\lim_{\lambda\to \infty}\lambda^{s-\frac{n}{2}}\mathcal{H}(\lambda,s)=\omega_n\cdot \mathsf{AVR}_g\cdot  \frac{\Gamma\left(\frac{n}{2}+1\right)\Gamma\left(s-\frac{n}{2}\right)}{\Gamma(s)},
	\]
	(see e.g. Krist\'aly \cite[Proposition 2.3/(ii)]{KristalyCalcVar24}) and the following  property of the gamma function: \[\Gamma(z+1)=z\Gamma(z),\quad \forall z>0.\]

	After taking the limit and simplifying the left hand side, we obtain 
	\begin{align*}
		(n-4)(n-2)n(n+2)\frac{\Gamma\!\left(\frac n2\right)}{\Gamma(n)}
		&\ge
		\frac{C}{\pi^2}\mathsf{AVR}_g^{-\frac{4}{n}}\left(\frac{\Gamma\left(\frac{n}{2}\right)}{\Gamma (n)}\right)^\frac{n-4}{n}.
	\end{align*}
	Observe that the letter inequality is equivalent to 
	\[
		S_{2,2}\cdot \mathsf{AVR}_g^{\frac{4}{n}}\ge C,
	\]
	which is a contradiction.  This completes the proof.
\end{proof}

\begin{remark}\label{rem:sharp}
	One may wonder whether similar (possibly higher order) assumptions as \(\rho\Delta_g\rho\ge n-5\) yield the sharpness of higher-order Sobolev inequalities. An important ingredient of the previous proof is to guarantee inequality~\eqref{eq:sovolev:sharp:pointwise}, which can be rewritten as
	\[
		(\Delta_gu_\lambda )^2\le L(\rho)^2,\quad\text{where}\quad L(|x|)=\Delta_{g_0}t_\lambda(|x|).
	\]
	It turns out that the higher-order version of the above pointwise inequality implies rigidity (the manifold needs to be Euclidean). Thus, a possible proof of the sharpness (if exists) may rely on non-pointwise estimates.
\end{remark}

\section{Rellich inequalities and improvements}\label{sec:rellich}
In this section we prove Theorem~\ref{thm:rellich}, that is we establish higher-order Rellich inequalities on Riemannian manifolds supporting inequality~\eqref{eq:intro:isoperimetric}. After that, we present various improved Rellich inequalities, which can be proved in a similar manner.

\begin{proof}[Proof of Theorem~\ref{thm:rellich}] 
	The proof follows a similar line of thoughts as the proof of Theorem~\ref{thm:sobolev}: We observe that it is enough to prove it for regular domains \(\Omega\). In addition, we define \(f=(-\Delta_g)^ku\) and consider the problem~\eqref{eq:itp:pde}, and its symmetrization~\eqref{eq:itp:pde:symm}, whose solution is denoted by \(v\).

	On the one hand, if \(m=2k\), by applying the Cavalieri principle (see~\eqref{eq:prelim:Cavalieri}), the iterated Talenti's principle (see Theorem~\ref{thm:itp}), the Euclidean Rellich inequality, and the Hardy--Littlewood inequality (see~\eqref{eq:prelim:HardyLittlewood}) we get
	\begin{align*}
		\int_\Omega|D_g^mu|^p\diff v_g&=\int_\Omega|\Delta_g^ku|^p\diff v_g=\int_\Omega |f|^p\diff v_g=\int_{\Omega^\star} (f^\star)^p\diff v_{g_0}=\int_{\Omega^\star}|\Delta_{g_0}^kv|^p\diff v_{g_0}\\
		&\ge R_{m,p}\cdot\int_{\Omega^\star}\frac{|v|^p}{|x|^{mp}}\diff v_{g_0} \ge (c_g)^{2kp}\cdot R_{m,p}\cdot\int_{\Omega^\star}\frac{(u^\star)^p}{|x|^{mp}} \diff v_{g_0}\\
		&\ge (c_g)^{mp}\cdot R_{m,p}\cdot\int_{\Omega}\frac{|u|^p}{\rho^{mp}} \diff v_{g}.
	\end{align*}

	On the other hand, if \(m=2k+1\), we also use the Pólya--Szegő inequality (see~\eqref{eq:prelim:PolyaSzego}) to obtain
	\begin{align*}
		\int_\Omega|D_g^mu|^p\diff v_g&=\int_\Omega|\nabla_g\Delta_g^ku|^p\diff v_g=\int_\Omega |\nabla_gf|^p\diff v_g\ge (c_g)^p\cdot \int_{\Omega^\star} (\nabla_{g_0}f^\star)^p\diff v_{g_0}\\
		&=(c_g)^p\cdot\int_{\Omega^\star}|\nabla_{g_0}\Delta_{g_0}^kv|^p\diff v_{g_0}\ge (c_g)^p\cdot R_{m,p}\cdot\int_{\Omega^\star}\frac{|v|^p}{|x|^{mp}}\diff v_{g_0} \\
		&\ge (c_g)^{(2k+1)p}\cdot R_{m,p}\cdot\int_{\Omega^\star}\frac{(u^\star)^p}{|x|^{mp}} \diff v_{g_0}\ge (c_g)^{mp}\cdot R_{m,p}\cdot\int_{\Omega}\frac{|u|^p}{\rho^{mp}} \diff v_{g},
	\end{align*}
	which concludes the proof.
\end{proof}
\begin{proof}[Proof of Theorem~\ref{thm:general}] Observe that the proof of Theorem~\ref{thm:rellich} provides a `recipe' for threating improved Rellich inequalities of the form \[
	\int_{\Omega^\star}|D_{g_0}^mu|^p\diff v_{g_0}\ge C_{m,p}\cdot \int_{\Omega^\star} |u|^p h(|x|)\diff v_{g_0},\quad u\in C_0^\infty(\Omega^\star).
\]
The only difference is that we need to use the Hardy--Littlewood inequality (see~\eqref{eq:prelim:HardyLittlewood}) for a generic \(h\) instead of \(t\mapsto \frac{1}{t^{mp}}\).
\end{proof}

In the rest of the paper, we list various applications of Theorem~\ref{thm:general}. For brevity, in each case, first we indicate a reference for the Euclidean version, next we state the Riemannian extension of the corresponding inequality.
\begin{itemize}[wide]
	\item  Adimurthi, Grossi, and Santrac~\cite[Theorem 2.1/(a)]{adimurthi2006optimal}: For any function \(h\) let 
	\(h_{[0]}(t)=t\), \(h_{[1]}(t)=h(t)\) and \( h_{[s]}(t)=h(h_{[s-1]}(t))\), for all \(s\ge 2\).
	Define \(\ell(t)=\frac{1}{1-\log(t)}\), and let \(T\ge 1\) be an integer and \(R>0\). Then for every  \(u\in C_0^\infty(\Omega)\) one has 
	\[
		\int_\Omega |\Delta_g u|^2\diff v_g\ge (c_g)^4\cdot\left(\frac{n^2(n-4)^2}{16}\int_\Omega \frac{u^2}{\rho^4}\diff v_g + \left(1+\frac{n(n-4)}{8}\right)\sum_{j=1}^T\int_\Omega \frac{u^2}{\rho^4}\prod_{s=1}^j\ell_{[s]}^2\left(\frac{\rho}{R}\right)\diff v_g\right).
    \]
	\item Ghoussoub and Moradifam~\cite[Corollary 3.8/1]{ghoussoub2011bessel}: Let \(n\ge 5\), \(\Omega\subseteq M\) be regular, \(k\ge 1\) be an integer, \(R\) denote the radius of \(\Omega^\star\) and define \(r=R\cdot \exp_{[k-1]}(e)\). Then for every \(u\in C_0^\infty(\Omega)\) one has 
	\[
		\int_\Omega |\Delta_g u|^2\diff v_g\ge (c_g)^4\cdot\left(\frac{n^2(n-4)^2}{16}\int_\Omega \frac{u^2}{\rho^4}\diff v_g + \left(1+\frac{n(n-4)}{8}\right)\sum_{j=1}^k\int_\Omega \frac{u^2}{\rho^4}\left(\prod_{i=1}^j\log_{[i]}\left(\frac{r}{\rho}\right)\right)^{-2}\diff v_g\right).
	\]
	\item Gazzola, Grunau, and Mitidieri~\cite[Theorem 2]{gazzola2004hardy}: Let \(p\ge 2\), \(n\ge 2p\), and \(\Omega\subset M\) be regular, then for every \(u\in C_0^\infty(\Omega)\) one has
	\begin{align*}
		\int_\Omega|\Delta_g u|^p\diff v_g&\ge (c_g)^{2p}\cdot \left(\frac{(n-2p)(p-1)n}{p^2}\right)^p\cdot\int_\Omega\frac{|u|^p}{\rho^{2p}}\diff v_g\\ 
		&\qquad +(c_g)^{2p}\cdot\frac{4(p-1)^p(n-2p)^{p-1}n^{p-1}\gamma }{p^{2p-1}}\cdot \left(\frac{\omega_n}{\operatorname{vol}_g(\Omega)}\right)^\frac{2}{n}\cdot \int_\Omega\frac{|u|^p}{\rho^{2p-2}}\diff v_g\\
		&\qquad+(c_g)^{2p}\cdot\Gamma\cdot \left(\frac{\omega_n}{\operatorname{vol}_g(\Omega)}\right)^\frac{2p}{n} \cdot \int_\Omega|u|^p\diff v_g.
	\end{align*}
	The constants appearing in the above inequality are defined as follows. Let \(\vartheta=4+\frac{n(p-2)}{p}\), introduce the function space \(X=\{v\in C^2([0,1]):v'(0)=v(1)=0, v\not\equiv 0\}\), and denote \(\Delta_\vartheta v(r)=v''(r)+\frac{\vartheta-1}{r}v'(r)\). Define 
	\[
		\gamma\defeq\displaystyle\inf_X\frac{\int_0^1 r|v'(r)|^2\diff r}{\int_0^1 r v(r)^2\diff r}\quad\text{and}\quad\Gamma\defeq\displaystyle\max\left\{\frac{(p-1)^{p-1}(n-2p)^{p-2}n^{p-2}}{p^{2(p-2)}}\Lambda_\vartheta,\lambda_\vartheta\right\},
	\]
	where \(\lambda_\vartheta\defeq\displaystyle \inf_X\frac{\int_0^1 r^{2p-1}|\Delta_\vartheta v(r)|^p\diff r}{\int_0^1 r^{2p-1}|v(r)|^p\diff r}\) and
	\(\Lambda_\vartheta\defeq\displaystyle\inf_X\frac{\int_0^1 r^3 |v(r)|^{p-2}(\Delta_\vartheta v(r))^2\diff r}{\int_0^1r^{2p-1}|v(r)|^p\diff r}\).

	In particular, when \(p=2\), we obtain the following second order Brezis--Vazquez-type inequality: 
	\begin{align*}
		\int_\Omega(\Delta_gu)^2\diff v_g&\ge (c_g)^{4}\cdot \frac{n^2(n-4)^2}{16}\cdot \int_\Omega \frac{u^2}{\rho^4}\diff v_g+ (c_g)^{4}\cdot\frac{n(n-4)}{2}\cdot \Lambda(2)\cdot\left(\frac{\omega_n}{\operatorname{vol}_g(\Omega)}\right)^\frac{2}{n}\cdot\int_B\frac{u^2}{\rho^2}\diff v_g\\
		&\qquad+(c_g)^{4}\cdot {\Lambda(4)^2}\cdot\left(\frac{\omega_n}{\operatorname{vol}_g(\Omega)}\right)^\frac{4}{n}\cdot \int_\Omega u^2\diff v_g,\quad\forall u\in C_0^\infty(\Omega),
	\end{align*}
	where \(\Lambda(n)=\displaystyle\inf_X\frac{\int_0^1r^{n-1}|v'(r)|^2\diff r}{\int_0^1r^{n-1}v(r)^2\diff r}\).
	\item Gazzola, Grunau, and Mitidieri~\cite[Corollary 2]{gazzola2004hardy}: Let \(\Omega\subset M\) be regular and \(m\ge 1\) be an integer with \(n>2m\). Then there exist some constants \(c_1,c_2, \dots, c_m\) depending only on \(m\), \(n\), and \(\operatorname{vol}_g(\Omega)\) such that for every \(u\in C_0^\infty(\Omega)\) one has
	\[\int_\Omega |D^mu|^2\diff v_g\ge (c_g)^{2m}\cdot \left(L_m\cdot \int_\Omega\frac{u^2}{\rho^{2m}}\diff v_g+\sum_{s=1}^m c_s\int_\Omega\frac{u^2}{\rho^{2(m-s)}}\diff v_g\right),\]
	where \(L_m=\displaystyle\frac{1}{4^m}\prod_{s=1}^m(n+2m-4j)^2\).
\end{itemize}

\medskip
\noindent\textbf{Acknowledgement.} The authors would like to thank to the anonymous reviewers for carefully reading the article and providing valuable comments.

\end{document}